\documentclass[11pt]{amsart}
\usepackage{amssymb}
\usepackage{amsmath}
\usepackage{graphicx}
\usepackage[latin1]{inputenc}
\usepackage{amsfonts}
\usepackage[spanish,english]{babel}
\usepackage{comma}
\usepackage[active]{srcltx}
\providecommand{\U}[1]{\protect\rule{.1in}{.1in}}
\newtheorem{theorem}{Theorem}
\theoremstyle{plain}

\newtheorem{corollary}[theorem]{Corollary}

\newtheorem{definition}[theorem]{Definition}

\newtheorem{proposition}[theorem]{Proposition}

\numberwithin{equation}{section}
\begin{document}
\title{\textbf{Pseudoprimes stronger than strong pseudoprimes}}
\author[J.H. Castillo]{John H. Castillo}
\address{John H. Castillo, Departamento de Matemáticas y Estadística, Universidad de Nariño, San Juan de Pasto-Colombia}
\email{jhcastillo@gmail.com}
\author[G. García-Pulgarín]{Gilberto Garc\'\i a-Pulgar\'in}
\address{Gilberto Garc\'\i a-Pulgar\'in, Universidad de Antioquia, Medellín-Colombia}
\email{gigarcia@ciencias.udea.edu.co}
\author[J.M Velásquez-Soto]{Juan Miguel Vel\'asquez-Soto}
\address{Juan Miguel Vel\'asquez-Soto, Departamento de Matemáticas, Universidad del Valle, Cali-Colombia}
\email{jumiveso@univalle.edu.co} 
\date{}

\begin{abstract}
We introduce a new class of pseudoprimes. In this work we
characterize Midy pseudoprimes, give some of their properties and
established interesting connections with other known pseudoprimes,
in particular we show that every divisor of a Midy pseudoprime is
either a prime or a Midy pseudoprime and in the last case it is a
strong pseudoprime.

\end{abstract}
\maketitle

\section{Midy's Property}

Let $b$ be a positive integer greater than $1$, $b$ will denote the
base of numeration, $N$ a positive integer relatively prime to $b$,
i.e $(N,b)=1$, $\left\vert b\right\vert _{N}$ the order of $b$ in
the multiplicative group $\mathbb{U}_{N}$ of positive integers less
than $N$ and relatively primes to $N,$ and $x\in\mathbb{U}_{N}$. It
is well known that when we write the fraction $\frac{x}{N}$ in base
$b$, it is periodic. By period we mean the smallest repeating
sequence of digits in base $b$ in such expansion, it is easy to see
that $\left\vert b\right\vert _{N}$ is the length of the
period of the fractions $\frac{x}{N}$ (see Exercise 2.5.9 in \cite{Nathanson}%
). Let $d,\,k$ be positive integers with $\left\vert b\right\vert
_{N}=dk$, $d>1$ and $\frac{x}{N}=0.\overline{a_{1}a_{2}\cdots
a_{\left\vert b\right\vert _{N}}}$ where the bar indicate the period
and $a_{i}$'s are digits in base $b$. We separate the period
${a_{1}a_{2}\cdots a_{\left\vert b\right\vert _{N}}}$ in $d$ blocks
of length $k$ and let
\[
A_{j}=[a_{(j-1)k+1}a_{(j-1)k+2}\cdots a_{jk}]_{b}%
\]
be the number represented in base $b$ by the $j$-th block and $S_{d}%
(x)=\sum\limits_{j=1}^{d}A_{j}$. If for all $x\in\mathbb{U}_{N}$,
the sum $S_{d}(x)$ is a multiple of $b^{k}-1$ we say that $N$ has
the Midy's property for $b$ and $d$. It is named after E. Midy
(1836), to read historical aspects about this property see
\cite{Lewittes} and its references.

We denote with $\mathcal{M}_{b}(N)$ the set of positive integers $d$
such that $N$ has the Midy's property for $b$ and $d$ and we will
call it the Midy's set of $N$ to base $b$. As usual, let
$\nu_{p}(N)$ be the greatest exponent of $p$ in the prime
factorization of $N$.

For example $13$ has the Midy's property to the base $10$ and $d=3$,
because $|13|_{10}=6$, $1/13=0.\overline{076923}$ and $07+69+23=99$.
Also, $75$ has the Midy's property to the base $8$ and $d=4$, since
$|75|_{8}=20$, $1/75=[0.\overline{00664720155164033235}]_8$ and
$[00664]_8+[72015]_8+[51640]_8+[33235]_8=2*(8^5-1)$. But $75$ does
not have the Midy's property to $8$ and $5$. Actually, we can see
that $\mathcal{M}_{10}(13)=\{2,3,6\}$ and $\mathcal{M}_{8}(75)=\{4,
20\}$.

In \cite{garcia09} is given the following characterization of Midy's
property.

\begin{theorem}
\label{ppl2} If $N$ is a positive integer and $\left\vert
b\right\vert _{N}=kd$, then $d\in\mathcal{M}_{b}(N)$ if and only if
$\nu_{p}(N)\leq\nu _{p}(d)$ for all prime divisor $p$ of $(b^{k}-1,\
N)$.
\end{theorem}

The next theorem is a different way to write Theorem \ref{ppl2}.

\begin{theorem}
\label{ppl3}Let $N$ be a positive integer and $d$ a divisor of
$\left\vert b\right\vert _{N}$. \ The following statements are
equivalent

\begin{enumerate}
\item $d\in\mathcal{M}_{b}(N)$

\item For each prime divisor $p$ of $N$ such that $\nu_{p}\left(  N\right)
>\nu_{p}\left(  d\right)  $, there exists a prime $q$ divisor of $\left\vert
b\right\vert _{N}$ that satisfies $\nu_{q}\left(  \left\vert
b\right\vert _{p}\right)  >\nu_{q}\left(  \left\vert b\right\vert
_{N}\right)  -\nu _{q}\left(  d\right)  $.
\end{enumerate}
\end{theorem}

In \cite{trio} the authors prove the following theorem.

\begin{theorem}
\label{coro}Let $d_{1}$, $d_{2}$ be divisors of $\left\vert
b\right\vert _{N}$ and assume that $d_{1}\mid d_{2}$ and
$d_{1}\in\mathcal{M}_{b}(N)$, then $d_{2}\in\mathcal{M}_{b}(N)$.
\end{theorem}

It is easy to see that if $N$ is a prime number, then any divisor of
$|b|_N$ greater than $1$ is an element of $\mathcal{M}_{b}(N)$. In
the next section, we will study when a given composite number $N$
satisfies the above property. To do that and by the last theorem it
is important to know when a prime divisor of $|b|_N$ is in
$\mathcal{M}_{b}(N)$. It was characterized by the authors in
\cite[Corollary 5]{dirichletmidy}. We recall that result here.

\begin{theorem}[\cite{dirichletmidy},Corollary 5]
\label{co1}Let $N$ be a positive integer and let $q$ be a prime
divisor of $\left\vert b\right\vert _{N}$, then
$q\in\mathcal{M}_{b}(N)\ $ if and only if

\begin{enumerate}
\item If $\left(  N,\ q\right)  =1$, then $\nu_{q}(\left\vert b\right\vert
_{p})=\nu_{q}\left(  \left\vert b\right\vert _{N}\right)  $ for all
$p$ prime divisor of $N$.

\item If $\left(  N,\ q\right)  >1$, then $q^{2}$ not divides $N$ and
$\ \nu_{q}(\left\vert b\right\vert _{p})=\nu_{q}\left(  \left\vert
b\right\vert _{N}\right)  $ for all $p$ prime divisor of $N$
different from $q$.
\end{enumerate}
\end{theorem}

\section{Midy pseudoprimes}

Pomerance and Crandall in their book \cite{prime}, state that:

\begin{quotation}
Suppose we have a theorem, \textquotedblleft\textit{If
}$\mathit{n}$\textit{
is prime, then S is true about }$\mathit{n}$\textit{,}\textquotedblright%
\ where \textquotedblleft S\textquotedblright\ is some easily
checkable arithmetic statement. If we are presented with a large
number $n$, and we wish to decide whether $n$ is prime or composite,
we may very well try out the arithmetic statement S and see whether
it actually holds for $n$. If the statement fails, we have proved
the theorem that $n$ is composite. If the statement holds, however,
it may be that $n$ is prime, and it also may be that $n$ is
composite. So we have the notion of S-pseudoprime, which is a
composite integer for which S holds.
\end{quotation}

Applying the above commentary, to the Fermat's little theorem the
concepts of pseudoprime and strong pseudoprime are given as follows

\begin{definition}
The composite integer $N$ is called a pseudoprime (or Fermat
pseudoprime) to
base $b$ if $\left(  b, N\right)  =1$ and $b^{N-1}\equiv1\ \operatorname{mod}%
\ N$. An integer which is pseudoprime for all possible bases $b$ is
called a Carmichael number or an absolute pseudoprime. An odd
composite $N$ such that $N-1=2^{r}s$ with $s$ an odd integer and
$\left(  b,\ N\right)  =1$, is said to be a strong pseudoprime to
base $b$~ if either $b^{s}\equiv 1\ \operatorname{mod}\ N$ or
$b^{2^{i}s}\equiv-1\ \operatorname{mod}\ N$, for some $0< i<r$.
\end{definition}

\begin{proposition}
\label{prop13} An odd composite integer $N$ is a strong pseudoprime
to base $b$ if and only if $N$ is pseudoprime to base $b$ and there
is a non-negative integer $k$ such that $\nu_{2}\left(  \left\vert
b\right\vert _{p^{\nu _{p}\left(  N\right)  }}\right)
=\nu_{2}\left(  \left\vert b\right\vert _{p }\right)  =k$ for all
prime $p$ divisor of $N$.
\end{proposition}

\begin{proof}
Let $N-1=2^{t}s$. Assume that $N$ is pseudoprime to base $b$ and
\linebreak$\nu_{2}\left(  \left\vert b\right\vert _{p^{\nu_{p}\left(
N\right)  }}\right)  =k$, for some non-negative integer $k$ and for
any prime divisor $p$ of $N$. If $k=0$, it follows that $\left\vert
b\right\vert _{N}$
is odd and as $b^{N-1}\equiv1\ \operatorname{mod}\ N$, then $b^{s}%
\equiv1\ \operatorname{mod}\ N$. If $k>0$, let $\left\vert
b\right\vert
_{p^{\nu_{p}\left(  N\right)  }}=2^{k}s_{p}$, then $b^{2^{k-1}s_{p}}%
\equiv-1\ \operatorname{mod}\ p^{\nu_{p}\left(  N\right)  }$ and
thus $b^{2^{k-1}s}\equiv-1\ \operatorname{mod}\ N$ for each prime
divisor $p$ of $N$. Therefore, in any case we obtain that $N$ is a
strong pseudoprime to base $b$. The reciprocal can be prove in a
similar way.
\end{proof}

The smallest absolute pseudoprime is $561$ and in general such
numbers are square-free and product of at least three primes, Alford
et al. in \cite{Carmichael} proved that there are infinitely many
absolute pseudoprimes.

Theorem \ref{ppl2} implies that if $N$ is prime then $N$ verifies
the Midy's property for any base $b$ and for all divisor $d$,
different from $1$, of $\left\vert b\right\vert _{N}$, this fact and
the commentary quoted from Pomerance and Crandall leave us to study
``Midy pseudoprimes" and we will dedicate the rest of this work to
do it.

\begin{theorem}
\label{seudo}If $N$ is a positive integer such that for all $d>1$
and divisor of $\left\vert b\right\vert _{N}$ is satisfied that
$d\in\mathcal{M}_{b}(N)$, then $\left(  N,\ \left\vert b\right\vert
_{N}\right)  $ is either $1$ or a prime.
\end{theorem}

\begin{proof}
Let $p_{1}<$ $p_{2}$ be prime divisors of $\left(  N, \left\vert
b\right\vert _{N}\right)  $. Write $N=p_{1}p_{2}N_{1}$ for some
integer $N_{1}$, therefore $\left\vert b\right\vert
_{N}=p_{1}p_{2}\left[  \left\vert b\right\vert _{p_{1}}, \left\vert
b\right\vert _{p_{2}}\right]  r$; with $r$ an integer. Let
$d=p_{2}$, $\left\vert b\right\vert _{N}=p_{2}k$, because
$\left\vert b\right\vert _{p_{1}}$ divides $k$ it follows that
$p_{1}\mid\left(  N,
b^{k}-1\right)  $ and since $d\in\mathcal{M}_{b}(N)$ we have $\nu_{p_{1}%
}\left(  N\right)  \leq\nu_{p_{1}}\left(  d\right)  $, which is a
contradiction as $1<\nu_{p_{1}}\left(  N\right)  $ and
$\nu_{p_{1}}\left( d\right)  =\nu_{p_{1}}\left(  p_{2}\right)  =0$.
It is clear that $p^{2}$ not divides $\left(  N, \left\vert
b\right\vert _{N}\right)  $.
\end{proof}

\begin{theorem}
\label{pseudomidy}Let $N$ be a positive integer, then
$d\in\mathcal{M}_{b}(N)$ for all divisor $d>1$ of $\left\vert
b\right\vert _{N}$, if and only if

\begin{enumerate}
\item If $\left(  N, \left\vert b\right\vert _{N}\right)  =1$, then
$N=p_{1}^{e_{1}}p_{2}^{e_{2}}\cdots p_{l}^{e_{l}}$ with each $p_{i}$
prime and $\left\vert b\right\vert _{N}=\left\vert b\right\vert
_{p_{i}}$ for $i=1,\cdots,l$.

\item If $\left(  N,\left\vert b\right\vert _{N}\right)  =r$ prime, then
$N=rp_{1}^{e_{1}}p_{2}^{e_{2}}\ldots p_{l}^{e_{l}}$ with each
$p_{i}$ prime
and $\left\vert b\right\vert _{N}=\left\vert b\right\vert _{p_{i}}%
=r^{s}\left\vert b\right\vert _{r}$ for $i=1,\ldots,l$ with $s$ a
positive integer.
\end{enumerate}
\end{theorem}

\begin{proof}
From Theorem \ref{coro}, it is clear that $d\in\mathcal{M}_{b}(N)$
for all divisor $d>1$ of $\left\vert b\right\vert _{N}$ if and only
if $q\in \mathcal{M}_{b}(N)$ for each prime divisor $q$ of
$\left\vert b\right\vert _{N}$. The part (1) is immediate from the
first case of Theorem \ref{co1}.

To prove the second part, assume that $d\in\mathcal{M}_{b}(N)$ for
all divisor $d>1$ of $\left\vert b\right\vert _{N}$. From Theorem
\ref{seudo} we have $r=\left(  N,\left\vert b\right\vert _{N}\right)
$ for some prime $r$. Take
$N=rN_{1}$ with $N_{1}$ an integer and $\left\vert b\right\vert _{N}%
=r^{s}\left\vert b\right\vert _{r}h$ where $(h,r)=1$. We will prove
that $h=1$. If there is a prime divisor $q$ of $h,$ from Theorem
\ref{co1} follows that $\nu_{q}\left(  \left\vert b\right\vert
_{N}\right) =\nu _{q}\left(  \left\vert b\right\vert _{p}\right)  $
for all prime divisor $p$ of $N,$ particulary to $p=r$ we obtain
$\nu_{q}\left(  \left\vert b\right\vert _{r}\right)  =\nu_{q}\left(
\left\vert b\right\vert _{N}\right)  =\nu _{q}\left(  \left\vert
b\right\vert _{r}\right)  +\nu_{q}\left(  h\right)  $ and hence
$\nu_{q}\left(  h\right)  =0$. Thus $h=1.$

Let $p\neq r$ a prime which divides $N$, we will see that
$\left\vert b\right\vert _{N}=\left\vert b\right\vert _{p}.$ Write
$\left\vert b\right\vert _{N}=\left\vert b\right\vert _{p}H$ with
$H$ an integer. If $p$ is a divisor of $H$, then $p$ divides
$\left\vert b\right\vert _{N}$ and consequently $p$ divides $\left(
N\text{, }\left\vert b\right\vert _{N}\right)  $ which is absurd
because $p\neq r$. Suppose that there exists a prime $q$ different
from $p$ and divisor of $H$, so $\left\vert b\right\vert
_{N}=\left\vert b\right\vert _{p}H=\left\vert b\right\vert
_{p}H_{1}q=qk$ and as, by the assumption, $q\in\mathcal{M}_{b}(N)$,
it leaves us to a contradiction from Theorem \ref{ppl2} because $p$
is a divisor of $\left( N,\text{ }b^{k}-1\right)  $. So $H=1$ and
$\left\vert b\right\vert _{p}=\left\vert b\right\vert
_{N}=r^{s}\left\vert b\right\vert _{r}.$ Therefore
$N=rp_{1}^{e_{1}}p_{2}^{e_{2}}\cdots p_{l}^{e_{l}}$ with each
$p_{i}$ prime and also $\left\vert b\right\vert _{N}=\left\vert
b\right\vert _{p_{i}}=r^{s}\left\vert b\right\vert _{r}$.

Conversely, assume that $N=rp_{1}^{e_{1}}p_{2}^{e_{2}}\cdots
p_{l}^{e_{l}}$, $\left(  N, \left\vert b\right\vert _{N}\right)  =r$
and $\left\vert b\right\vert _{N}=\left\vert b\right\vert
_{p_{i}}=r^{s}\left\vert b\right\vert _{r}$. Take $d>1$ a divisor of
$\left\vert b\right\vert _{N}$, $\left\vert b\right\vert _{N}=kd$
and let $g=\left(  N\text{, }b^{k}-1\right) $. Since $\left\vert
b\right\vert _{p_{i}}=\left\vert b\right\vert _{N}=kd$ for each $1
\leq i\leq l$, we obtain that $\left\vert b\right\vert _{p_{i}}$ is
not a divisor of $k$ therefore $p_{i}$ does not divide $g$. Thus
either $g=1$ or $g=r$. In any case, by Theorem \ref{ppl2}, $N$ has
the Midy's property for $b$ and $d$.
\end{proof}

\begin{definition}
We say that a number $N$ is a Midy pseudoprime to base $b$ if $N$ is
an odd composite number relatively prime to both $b$ and $\left\vert
b\right\vert _{N}$ and for all divisor $d>1$ \ of $\left\vert
b\right\vert _{N}$ we get that $d\in\mathcal{M}_{b}(N)$.
\end{definition}

By this definition the first part of Theorem \ref{pseudomidy} can be
write in the following way.

\begin{theorem}
\label{seudomidy}An odd composite number
$N=p_{1}^{e_{1}}p_{2}^{e_{2}}\ldots p_{l}^{e_{l}}$, with $p_{i}$'s
different primes and $N$ relatively prime to $b$, is a Midy
pseudoprime to base $b$ if and only if $\left\vert b\right\vert
_{N}=\left\vert b\right\vert _{p_{i}}$ for every $1\leq i \leq l$.
\end{theorem}

V. Shevelev defines in \cite{overpseudoprimes} the concept of
overpseudoprime numbers and characterized them in Theorem 7. That
result is equivalent to our Theorem \ref{seudomidy}, so the concepts
of overpseudoprime and Midy's pseudoprime agree.

Theorem \ref{seudomidy} give us the following characterization.

\begin{corollary}
\label{masfuerte}An odd composite $N$ is a Midy pseudoprime to base
$b$ if and only if each divisor of $N$ is either a prime or a Midy
pseudoprime to base $b$.
\end{corollary}

The bellow result, Theorem 2.3 of \cite{Motose2}, allows us to give
a equivalent form of Theorem \ref{seudomidy}. We denote with
$\Phi_{n}\left( x\right)  $ the $n$-th cyclotomic polynomial.

\begin{theorem}
[Theorem 2.3 of \cite{Motose2}]\label{motose}Let $m$, $b\geq2$,
$n\geq3$ and $p$ be integers, where $p$ is the greatest prime
divisor of $n$. Then a composite number $m$ is a divisor of
$\Phi_{n}\left(  b\right)  $ if and only if $b^{n}\equiv1\
\operatorname{mod}\ m$ and every prime divisor $q$ of $m$ satisfies
that
\[
n=\left\{
\begin{array}
[c]{ccc}%
\left\vert b\right\vert _{q} &  & \text{if }q\neq p,\\
&  & \\
p^{e}\left\vert b\right\vert _{p} &  & \text{if }q=p.
\end{array}
\right.
\]

\end{theorem}

The following result is a consequence of Theorems \ref{seudomidy}
and \ref{motose}.

\begin{theorem}
A composite number $N$ with $\left(  N,\left\vert b\right\vert
_{N}\right) =1,$ is a Midy pseudoprime to base $b$ if and only if
$\Phi_{\left\vert b\right\vert _{N}}\left(  b\right)  \equiv0\
\operatorname{mod}\ N$ and $\left\vert b\right\vert _{N}>1$.
\end{theorem}

Theorem 1 of \cite{seudoprimos_to_25} shows the subsequent result
for strong pseudoprimes. We present here a more wide version which
is a direct consequence of Theorems \ref{seudomidy} and
\ref{motose}.

\begin{theorem}
Let $N>2$ and $f_{N}\left(  b\right)  =\frac{\Phi_{N}\left(
b\right) }{\left(  N,\ \Phi_{N}\left(  b\right)  \right)  }$. If
$f_{N}\left( b\right)  $ is composite, then $f_{N}\left(  b\right)
$ is a Midy pseudoprime to base $b$.
\end{theorem}

Our next result extends Theorem 3.5.10 of \cite{prime} .

\begin{theorem}
Let $p$ be an odd prime and $1<b<p-1$, then $N=\frac{b^{p}+1}{b+1}$
is either a Midy pseudoprime to base $b$ or a prime.
\end{theorem}

\begin{proof}
It is well known that $n$ odd implies that $\Phi_{2n}\left(
b\right) =\Phi_{n}\left(  -b\right)  $ and from here
$N=\frac{b^{p}+1}{b+1}=\Phi _{p}\left(  -b\right)  =\Phi_{2p}\left(
b\right)  $. In consequence, $N$ is odd and congruent with $1\
\operatorname{mod}\ p$. Therefore, $\left( 2p,\Phi_{2p}\left(
b\right)  \right)  =1$ and the result follows from the last theorem.
\end{proof}

The set of bases of Midy pseudoprimality is closed respect to
powers, although it is not closed by product as we can see when take
$N=91$ which is Midy pseudoprime to bases $9$ and $16$ but it is not
to $53$, their product modulo $N$.

\begin{theorem}
If $N$ is a Midy pseudoprime to base $b$, then $N$ is Midy
pseudoprime to base $b^{t}$ for any positive integer $t\geq1$.
\end{theorem}

\begin{proof}
The result is immediate from Theorem \ref{seudomidy}, since $N$ is
Midy pseudoprime to base $b$ so $\left\vert b\right\vert
_{N}=\left\vert b\right\vert _{p}$ for each prime divisor $p$ of
$N.$ Now $\left\vert b^{t}\right\vert _{N}=\frac{\left\vert
b\right\vert _{N}}{\left(
t,\ \left\vert b\right\vert _{N}\right)  }=\frac{\left\vert b\right\vert _{p}%
}{\left(  t,\ \left\vert b\right\vert _{p}\right)  }=$ $\left\vert
b^{t}\right\vert _{p}$. It shows that $N$ is a Midy pseudoprime to
base $b^{t}$.
\end{proof}

\begin{theorem}
If $N$ is a Midy pseudoprime to base $b$, then $N$ is a pseudoprime
to base $b$.
\end{theorem}

\begin{proof}
Write $N=p_{1}^{e_{1}}p_{2}^{e_{2}}\ldots p_{l}^{e_{l}}$ and assume
that $N$ is a Midy pseudoprime to base $b$. From Theorem
\ref{seudomidy} follows $\left\vert b\right\vert _{N}=\left\vert
b\right\vert _{p_{i}}=$ $\left\vert b\right\vert _{p_{i}^{e_{i}}}=t$
for each $i=1,2,\ldots,l$. By the assumption
we get that $t\mid p_{i}-1$ for each $i=1,2,\ldots,l$ and thus $b^{p_{j}%
-1}\equiv1\ \operatorname{mod}\ p_{i}^{e_{i}}$ for all pair $i,j$.
So, $b^{p_{j}}\equiv b\ \operatorname{mod}\ p_{i}^{e_{i}}$ and from
here $b^{p_{j}^{e_{j}}}\equiv b\ \operatorname{mod}\ p_{i}^{e_{i}}$
and, consequently, $b^{p_{1}^{e_{1}}p_{2}^{e_{2}}\ldots
p_{l}^{e_{l}}}\equiv b\ \operatorname{mod}\ p_{i}^{e_{i}}$, namely
$b^{N}\equiv b\ \operatorname{mod}\ p_{i}^{e_{i}}$ for each $i$ and
therefore $b^{N}\equiv b\ \operatorname{mod}\ N$.
\end{proof}

\begin{theorem}
\label{midypseudoprimeisstrong}If $N$ is a Midy pseudoprime to base
$b$, then $N$ is a strong pseudoprime to base $b$.
\end{theorem}

\begin{proof}
Write $N=p_{1}^{e_{1}}p_{2}^{e_{2}}\ldots p_{l}^{e_{l}}$ and assume
that $N$ is a Midy pseudoprime to base $b$. We know that $N$ is a
pseudoprime to base $b$. Since $N$ is a Midy pseudoprime to base $b$
implies that $\left\vert b\right\vert _{N}=\left\vert b\right\vert
_{n}$ for each divisor $n$ of $N$ and thus there is a non-negative
integer $k$ such that for all prime divisor $p$ of $N$ we get that
$\nu_{2}\left(  \left\vert b\right\vert _{p^{\nu _{p}\left(
N\right)  }}\right)  =k$ and the result follows from Proposition
\ref{prop13}.
\end{proof}

The reciprocal is not true. For example $N=91$ is a strong
pseudoprime to base $53$, but this is not a Midy pseudoprime to this
base.

Additionally, from the last theorem and Corollary \ref{masfuerte} we
get that every composite divisor of a Midy pseudoprime is a strong
pseudoprime, in this sense the Midy pseudoprimes are stronger than
strong pseudoprimes.

Among the first $58892$ strong pseudoprimes to base $2$ there are
only $31520$ Midy pseudoprimes to base $2$. Similarly, to base $3$
there are  $2558$ Midy pseudoprimes in the first $6087$ strong
pseudoprimes  and we found $582$ Midy pseudoprimes to base $5$ in
the first $1288$ strong pseudoprime to base $5$. Almost the $47\%$
of the strong pseudoprimes are Midy pseudoprimes.

We denote with $\psi_{k}$ and $\widetilde{\psi}_{k}$ the smallest
strong pseudoprime and the smallest Midy pseudoprime to all the
first $k$ primes taken as bases, respectively. From Theorem
\ref{midypseudoprimeisstrong} we know that
$\psi_{k}\leq\widetilde{\psi}_{k}$ for every positive integer $k$.
With some calculations, we can see that $\widetilde{\psi}_{1}=2047$,
$\widetilde{\psi}_{2}=5173601$ and $\widetilde{\psi}_{3}=960946321$.
We know, by \cite{Jaeschke}, the exact values for $\psi_k$, with
$1\leq k\leq 8$. Thus, $\widetilde{\psi}_{4}>3215031751=\psi_{4}$,
$\widetilde{\psi}_{5}>2152302898747=\psi_{5}$, $
\widetilde{\psi}_{6}>3474749660383=\psi_{6}$, $\widetilde{\psi}_{7}>341550071728321=\psi_{7}$ and $\widetilde{\psi}_{8}%
>341550071728321=\psi_{8}$.

\section*{Acknowledgements}
The authors are members of the research group: \'Algebra, Teor\'ia
de N\'umeros y Aplicaciones, ERM. J.H. Castillo was partially
supported by CAPES, CNPq from Brazil and Universidad de Nariño from
Colombia. J.M. Velásquez-Soto was partially supported by CONICET
from Argentina and Universidad del Valle from Colombia.

\bibliographystyle{amsalpha}
\bibliography{bibliografiaggp}

\end{document}